\documentclass[11pt]{amsart}

\usepackage{amsmath, amssymb, amsthm, mathtools}
\usepackage{mathrsfs}
\usepackage{enumerate}
\usepackage[hidelinks]{hyperref}
\usepackage{geometry}
\theoremstyle{plain}
\newtheorem{theorem}{Theorem}[section]
\newtheorem{corollary}[theorem]{Corollary}
\newtheorem{proposition}[theorem]{Proposition}
\newtheorem{lemma}[theorem]{Lemma}
\newtheorem{conjecture}[theorem]{Conjecture}
\newtheorem{example}[theorem]{Example}

\theoremstyle{definition}
\newtheorem{definition}[theorem]{Definition}

\theoremstyle{remark}
\newtheorem{remark}[theorem]{Remark}

\newcommand{\D}{\mathbb{D}}
\newcommand{\B}{\mathbb{B}}
\newcommand{\Bn}{\mathbb{B}_n}
\newcommand{\Sn}{\mathbb{S}^{2n-1}}
\newcommand{\Sph}{\mathbb{S}}
\newcommand{\R}{\mathbb{R}}
\newcommand{\C}{\mathbb{C}}

\newcommand{\dd}{\,\mathrm{d}}

\DeclareMathOperator{\Aut}{Aut}

\newcommand{\Ap}[2]{A^{#1}_{#2}}
\newcommand{\Hp}[1]{H^{#1}}

\newcommand{\dm}{\,\mathrm{d}m}

\title[Contraction via isoperimetry on the Bergman ball]{Contraction properties for holomorphic functions via isoperimetric stability on the Bergman ball}
\author{David Kalaj}
\address{University of Montenegro, Faculty of Natural Sciences and Mathematics, Podgorica, Montenegro}
\email{davidkalaj@gmail.com}
\author{Jian-Feng Zhu}
\address{Department of Mathematics, Shantou University, Shantou, Guangdong, China}
\email{flandy@stu.edu.cn}
\date{\today}
\hypersetup{
  pdftitle={Contraction properties for holomorphic functions via isoperimetric stability on the Bergman ball},
  pdfauthor={David Kalaj and Jian-Feng Zhu}
}

\begin{document}
\begin{abstract}
We prove a local distribution comparison for holomorphic functions that are
nearly constant on the unit ball. The argument converts quantitative geometric
information on weighted superlevel sets into truncated analytic deficit
inequalities and rests on a Fuglede-type stability estimate for nearly
spherical sets in the Bergman ball. Along the line
$q/p=\beta/\alpha$, the comparison controls the contribution of a compact
interval of levels to the layer-cake moments and quantitatively controls the
deviation of the recentered level sets from geodesic spheres. We also identify
the Hardy endpoint of the normalized weighted Bergman scale by an Abelian
distribution argument. The conclusions are deliberately local in the level
variable: no estimate for the uncontrolled layer-cake tails is asserted.
\end{abstract}
\maketitle


\section{Introduction}
\subsection{Introduction and main perspective}

A central theme of this work is the relationship between analytic norms of
holomorphic functions on the unit ball and the geometry of their level sets
with respect to the invariant (Bergman) measure. In particular, extremal
inequalities for weighted Bergman and Hardy norms are closely tied to sharp
isoperimetric properties of sublevel and superlevel sets of nonlinear
functionals of the form
\[
z \longmapsto |f(z)|^{p}(1-|z|^2)^{\alpha}.
\]
A guiding principle in geometric--analytic inequalities is that extremizers are
highly symmetric and that near-extremizers inherit quantitative symmetry. Our
approach links contraction phenomena for holomorphic function spaces to a
geometric analysis of weighted superlevel sets and, more specifically, to
quantitative isoperimetric stability in the Bergman ball. In the
one-dimensional setting, work of Kulikov shows that broad classes of Hardy and
Bergman functionals are maximized at normalized reproducing kernels, leading to
sharp contractive embeddings. Related contraction results in higher dimensions
appear in several forms; see, for instance,
\cite{KalajJFA2024LogM,KulikovGAFA2022,LiSuCHB2024}.

The constant function $f\equiv1$ plays a distinguished role: it is an
extremizer for a wide class of sharp inequalities, and its weighted superlevel
sets are centered Bergman balls. The purpose of the present paper is to study
the perturbative regime $f=1+\phi$. We show that, on a compact interval of
regular levels for which the superlevel sets remain nearly spherical, the
coarea inequality together with a Fuglede-type remainder yields both a
one-sided distribution comparison and a quantitative shape deficit.

The first main results, Theorems~\ref{thm:local_dominance} and
\ref{thm:local_gap}, establish this local distribution principle. For
holomorphic functions sufficiently close to $1$ in a $C^1$ sense on a compact
subball, under geometric assumptions on the weighted superlevel sets---a
barycenter normalization, a radial-parameter cap, and a small radial-graph
condition---they yield (i) distribution dominance of the superlevel volumes on
a compact level interval and (ii) a quantitative estimate in which the
truncated moment deficit controls a coercive shape functional measuring
deviation from spherical symmetry. Dropping the nonnegative remainder in
Theorem~\ref{thm:local_gap} gives the truncated Bergman--Bergman comparison in
Corollary~\ref{cor:local_contraction}. Appendix~\ref{app:nearly_spherical_verification}
verifies the geometric hypotheses for sufficiently small perturbations,
subject to the stated weighted tail condition.

Building on Proposition~\ref{prop:bergman_to_hardy_limit},
Theorem~\ref{thm:local_chain} connects the normalized weighted Bergman scale to
the Hardy endpoint in the critical limit. The critical-limit proposition is
proved by an Abelian distribution argument combined with the Hardy radial
maximal theorem. As $\alpha\downarrow n$, the normalized outer weight detects
the small-level asymptotic of $|f(z)|^{rn}(1-|z|^2)^n$ and recovers the boundary
$H^{rn}$ average. This limiting statement is logically separate from the local
level-set comparison.

The local results are motivated by the convex-order conjecture of
Nicola--Riccardi--Tilli \cite[Conjecture~5.1]{NicolaRiccardiTilli2025}. For the
invariant infinite-measure formulation used here, it is convenient to record
the following finite version.

\begin{conjecture}\cite[Conjecture~5.1]{NicolaRiccardiTilli2025}
\label{conj:convex_functional}
Let $\Phi:[0,1]\to\mathbb{R}$ be convex with $\Phi(0)=0$, and assume that the
two integrals below are finite. For every $p\in(0,\infty)$, $\alpha>n$, and
every holomorphic function $f\in A^p_\alpha(\B_n)$ normalized by
\[
\|f\|_{A^p_\alpha}=1,
\]
one has
\begin{equation*}\label{eq:convex_functional_ineq}
\int_{\B_n}
\Phi\!\left(|f(z)|^p (1-|z|^2)^\alpha\right)\,dv_g(z)
\ \le\
\int_{\B_n}
\Phi\!\left((1-|z|^2)^\alpha\right)\,dv_g(z).
\end{equation*}
\end{conjecture}

The compact-interval comparison proved below should be viewed as a local,
quantitative mechanism related to this conjecture, rather than as a global
proof of it. Indeed, the hinge-function reduction in
Subsection~\ref{sec:hinge_limitation} shows explicitly that a global
convex-order conclusion requires control of all level ranges, while the
present perturbative argument controls only a prescribed compact interval.
This distinction is important for the logical independence of the results in
this paper from later global developments.

\section{Preliminaries}
\subsection{Hardy and Bergman preliminaries on the disk}

Let $\D=\{z\in\C:|z|<1\}$. For $\alpha>1$ and $0<p<\infty$, we say that a function $f$ in $\D$ belongs to the Bergman space $\Ap{p}{\alpha}$ if
\[
\|f\|_{\Ap{p}{\alpha}}^p := \int_{\D} (\alpha-1) |f(z)|^p (1-|z|^2)^{\alpha-2}\,\frac{\dd x\dd y}{\pi}<\infty.
\]
For $0<r<\infty$, we say that a function $f$ analytic in $\D$ belongs to the Hardy space $\Hp{r}$ if
\[
\|f\|_{\Hp{r}}^r := \sup_{0<\rho<1}\int_0^{2\pi} |f(\rho e^{i\theta})|^r\frac{\dd\theta}{2\pi}<\infty.
\]

We write
\[
dm(z):=\frac{dx\,dy}{\pi(1-|z|^2)^2}
\]
for the M\"obius-invariant area measure on $\D$.

We recall an extremal principle (in the spirit of Kulikov) for convex functionals on Bergman spaces.

\begin{theorem}[\cite{KulikovGAFA2022} Extremizers in the disk Bergman space]\label{thm:BergmanDisk}
Let $G:[0,\infty)\to\R$ be convex. Then the maximum of
\begin{equation*}\label{eq:bergmanineq}
\int_{\D} G\bigl(|f(z)|^p(1-|z|^2)^\alpha\bigr)\,\dm(z)
\end{equation*}
over all $f\in \Ap{p}{\alpha}$ with $\|f\|_{\Ap{p}{\alpha}}=1$
is attained at constant functions (equivalently, at normalized reproducing kernels via automorphisms).
\end{theorem}

Applying $G(t)=t^s$ with $s>1$ yields the contractive embedding chain:

\begin{corollary}[\cite{KulikovGAFA2022} Disk contraction chain]\label{cor:chain}
For $0<p<q<\infty$ and $1<\alpha<\beta<\infty$ with $\frac{p}{\alpha}=\frac{q}{\beta}=r$, every analytic $f$ satisfies
\[
\|f\|_{\Ap{q}{\beta}} \le \|f\|_{\Ap{p}{\alpha}} \le \|f\|_{\Hp{r}}.
\]
Equality holds for constants (equivalently, for normalized reproducing kernels).
\end{corollary}

\begin{remark}
The M\"obius group acts on $\Ap{p}{\alpha}$ by
\[
g(z)=f\!\left(\frac{z-\overline{ w}}{1-zw}\right)\frac{(1-|w|^2)^{\alpha/p}}{(1-zw)^{2\alpha/p}},\qquad w\in\D,
\]
preserving both $\|\,\cdot\,\|_{\Ap{p}{\alpha}}$ and the distribution of $|f(z)|^p(1-|z|^2)^\alpha$.
\end{remark}


\subsection{Hardy and weighted Bergman spaces on the unit ball}

Let $\B_n:=\{z\in\C^n:\ |z|<1\}$ denote the unit ball and let $\Sph:=\partial\B_n$
be its boundary. We write $d\sigma$ for the normalized surface measure on
$\Sph$ and $d\nu$ for the Lebesgue volume measure on $\B_n$.

\subsection{Hardy spaces}
A holomorphic function $f:\B_n\to\C$ belongs to the Hardy space $H^p(\B_n)$,
$0<p<\infty$, if
\begin{equation*}\label{eq:Hp_def}
\|f\|_{H^p(\B_n)}^p
:=
\sup_{0<r<1}\int_{\Sph} |f(r\zeta)|^p\,d\sigma(\zeta)
<\infty.
\end{equation*}
For $f\in H^p(\B_n)$ the radial limits $f^*(\zeta)=\lim_{r\uparrow1}f(r\zeta)$
exist for a.e.\ $\zeta\in\Sph$, and one has
\[
\|f\|_{H^p(\B_n)}^p=\int_{\Sph}|f^*(\zeta)|^p\,d\sigma(\zeta).
\]
For $p\ge 1$ this defines a norm, while for $0<p<1$ it is a quasi-norm.

\subsection{Invariant (Bergman) measure and normalized weighted Bergman spaces}
Let $\B_n:=\{z\in\C^n:\ |z|<1\}$ and let $d\nu$ denote Lebesgue volume measure on
$\B_n$. We define the invariant (Bergman) volume element by
\begin{equation*}\label{eq:bergman_measure}
dv_g(z):=(1-|z|^2)^{-(n+1)}\,d\nu(z).
\end{equation*}
Notice that
$\int_{\B_n}dv_g=\infty$. This measure is invariant under holomorphic automorphisms
of the ball: for every $\varphi\in\Aut(\B_n)$ and every integrable $F$,
\begin{equation*}\label{eq:bergman_invariance}
\int_{\B_n} F(\varphi(z))\,dv_g(z)=\int_{\B_n} F(z)\,dv_g(z).
\end{equation*}

Fix $\alpha>n$ and $0<p<\infty$. It is convenient to normalize the weighted
measure
\begin{equation*}\label{eq:weighted_prob_measure}
d\mu_\alpha(z):=c_{\alpha,n}\,(1-|z|^2)^{\alpha}\,dv_g(z),
\qquad
c_{\alpha,n}:=\Bigl(\int_{\B_n}(1-|z|^2)^{\alpha}\,dv_g(z)\Bigr)^{-1},
\end{equation*}
so that $\mu_\alpha(\B_n)=1$. The weighted Bergman space $A_\alpha^p(\B_n)$
consists of all holomorphic functions $f:\B_n\to\C$ such that
\begin{equation}\label{eq:Ap_alpha_def}
\|f\|_{A_\alpha^p(\B_n)}^p
:=
\int_{\B_n} |f(z)|^p\,d\mu_\alpha(z)
=
c_{\alpha,n}\int_{\B_n} |f(z)|^p (1-|z|^2)^\alpha\,dv_g(z)
<\infty.
\end{equation}
With this convention,
\begin{equation*}\label{eq:norm_one_is_one}
\|1\|_{A_\alpha^p(\B_n)}=1 \qquad\text{for all }\alpha>n \text{ and } p>0.
\end{equation*}
When $p\ge1$ this defines a norm; for $0<p<1$ it is a quasi-norm. In terms of
Lebesgue measure $d\nu$, \eqref{eq:Ap_alpha_def} is equivalently
\begin{equation*}\label{eq:Ap_alpha_dnu_equiv}
\|f\|_{A_\alpha^p(\B_n)}^p
=
c_{\alpha,n}\int_{\B_n} |f(z)|^p (1-|z|^2)^{\alpha-(n+1)}\,d\nu(z).
\end{equation*}

\begin{lemma}[Sharp pointwise estimate]\label{lem:sharp_pointwise}
For every $f\in A_\alpha^p(\B_n)$ and $a\in\B_n$,
\begin{equation}\label{eq:sharp_pointwise_bound}
 |f(a)|^p(1-|a|^2)^\alpha\le \|f\|_{A^p_\alpha}^p.
\end{equation}
\end{lemma}

\begin{proof}
Let $\varphi_a$ be the involutive automorphism interchanging $0$ and $a$.
Choose the holomorphic branch in the zero-free denominator and set
\[
 g_a(z):=\frac{(1-|a|^2)^{\alpha/p}}
 {(1-\langle z,a\rangle)^{2\alpha/p}}\,f(\varphi_a(z)).
\]
The sub-mean inequality at the origin, followed by radial integration against
the probability measure $d\mu_\alpha$, gives
$|g_a(0)|^p\le\|g_a\|_{A^p_\alpha}^p$. The automorphism change of variables
gives $\|g_a\|_{A^p_\alpha}=\|f\|_{A^p_\alpha}$, while
$|g_a(0)|^p=(1-|a|^2)^\alpha|f(a)|^p$.
\end{proof}

\subsection{The Bergman ball: metric, measure, and perimeter of level sets}

Let $\Bn=\{z\in\C^n:|z|<1\}$ be the unit ball endowed with the Bergman metric.
Here we borrow some notation and formulas from
\cite{KalajBergman2026Published,KalajBergman2026Arxiv}. We fix the normalization
of the Bergman metric, its volume and hypersurface forms, and the invariant
Laplacian exactly as in \cite{KalajBergman2026Arxiv}. In particular,
$\widetilde\Delta\log(1-|z|^2)=-4n$ and the induced coarea and perimeter
measures use this same metric normalization. We use
the invariant measure already introduced above,
\begin{equation*}\label{eq:bergman_measure_metric}
d\mu(z):=dv_g(z)=(1-|z|^2)^{-n-1}d\nu(z),
\end{equation*}
where $d\nu$ is Lebesgue measure on $\R^{2n}$. Thus $d\mu$ denotes the
invariant measure, whereas the normalized weighted probability measure is
always written $d\mu_\alpha$.
With the distance convention of \cite{KalajBergman2026Arxiv},
$d_b(0,z)=\operatorname{arctanh}|z|$. Accordingly, the parameter $r$ in
$|z|=\tanh(r/2)$ equals twice the distance from the origin. We retain this
radial parameter because it is the convention used in the stability theorem
below.

Let $U$ be a smooth real-valued function on $\Bn$ and assume the level set
\[
M=\{z\in\Bn:\ U(z)=c\}
\]
is compactly contained in $\Bn$. Denote by $\nabla_b U$ the gradient with respect to the Bergman metric and by $|\cdot|_b$ the Bergman norm.

\subsection{A coarea-type perimeter formula}

Let $\sigma_g$ denote the Riemannian volume form for the Bergman metric and let $\sigma_{\tilde g}$ be the induced hypersurface form on $M$. If $N$ is the unit normal along $M$ (in the Bergman metric), then (see e.g.\ Lee's formula)
\begin{equation}\label{eq:tilde_sigma}
\sigma_{\tilde g} = \iota_M^\ast(N\lrcorner \sigma_g).
\end{equation}

A direct computation shows
\begin{equation}\label{eq:formu}
\Bigl\langle \frac{\nabla_b U}{|\nabla_b U|_b}, \frac{\nabla U}{|\nabla U|}\Bigr\rangle
=\sqrt{1-|z|^2}\,\sqrt{1-\Bigl|\Bigl\langle \frac{\nabla U}{|\nabla U|},z\Bigr\rangle\Bigr|^2},
\end{equation}
and the tangent spaces agree as real hyperplanes:
\[
V\in T_zM \text{ (Bergman)}\iff V\in T_zM \text{ (Euclidean)}.
\]

Combining \eqref{eq:tilde_sigma} and \eqref{eq:formu} yields the key perimeter formula.

\begin{theorem}[\cite{KalajBergman2026Arxiv} Perimeter of a level set in the Bergman ball]\label{thm:perimeter_formula}
Let $E=\{z\in\Bn:\ U(z)<c\}$ with $\partial E=\{U=c\}$ smooth and compactly supported in $\Bn$.
Then
\begin{equation}\label{eq:forperimeter}
P(E)=\int_{U(z)=c}(1-|z|^2)^{-n-\frac12}
\sqrt{1-\Bigl|\Bigl\langle \frac{\nabla U}{|\nabla U|},z\Bigr\rangle\Bigr|^2}\,\dd\mathcal{H}(z),
\end{equation}
where $\mathcal{H}$ denotes the Euclidean $(2n-1)$-dimensional Hausdorff measure.
\end{theorem}

\begin{remark}
For the centered model ball $\mathbb{B}_r=\{z:|z|<t=\tanh(r/2)\}$, where $r$
is the radial parameter just described, \eqref{eq:forperimeter} reduces to the
classical expression computed by Zhu \cite{ZhuUnitBall2005}:
\begin{equation*}\label{eq:br_zhu}
P(\mathbb{B}_r)=\int_{\{ |z|=t\}}(1-t^2)^{-n}\,\dd\mathcal{H}
=t^{2n-1}(1-t^2)^{-n}\int_{\Sph}d\mathcal H(\omega).
\end{equation*}
Indeed, for $U(z)=|z|$ we have $\langle \nabla U/|\nabla U|,z\rangle=|z|$.
\end{remark}

\subsection{Fuglede-type stability in the Bergman ball}

We state a quantitative stability estimate for nearly spherical sets in the Bergman ball, in the spirit of Fuglede.

\begin{theorem}[\cite{KalajBergman2026Arxiv} Fuglede-type theorem in the Bergman ball]\label{thm:fuglede_bergman}
For every $r_0>0$ there exists $\varepsilon_0\in(0,1/2]$, depending only on $r_0$, such that the following holds.

Let $E\subset \Bn$ have Bergman barycenter at the origin and satisfy the volume constraint
\begin{equation*}\label{eq:vol_constraint_thm}
\mu(E)=\mu(\mathbb{B}_r)
\end{equation*}
for some $r\in(0,r_0]$. Assume $\partial E$ is a radial graph over the unit sphere $\Sph$:
\begin{equation}\label{eq:graph_param_thm}
Z(\omega)=\omega\,\tanh\!\left(\frac r2(1+u(\omega))\right),\qquad \omega\in\Sph,
\end{equation}
with
\begin{equation*}\label{eq:u_small_thm}
\|u\|_{W^{1,\infty}(\Sph)}\le \varepsilon_0.
\end{equation*}
Then there exists $c_1=c_1(r_0,n)>0$ such that
\begin{equation*}\label{eq:fuglede_estimate}
\frac{P(E)-P(\mathbb{B}_r)}{P(\mathbb{B}_r)}\ge c_1\,\|u\|_{W^{1,2}(\Sph)}^2.
\end{equation*}
In particular $P(E)\ge P(\mathbb{B}_r)$ with equality iff $E=\mathbb{B}_r$.
\end{theorem}

\begin{remark}
Theorem~\ref{thm:fuglede_bergman} is the quantitative geometric input used below: on the nearly spherical class, the perimeter deficit controls the size of the radial perturbation.
\end{remark}

\section{Main results}

\subsection{Distribution functions and the coarea inequality}\label{sec:distribution_coarea}

Fix $p>0$ and $\alpha>n$. For holomorphic $f:\B_n\to\C$, define
\begin{equation*}\label{eq:u_ball}
u_f(z):=|f(z)|^p(1-|z|^2)^\alpha,
\qquad
A_t(f):=\{z\in\B_n:\ u_f(z)>t\},
\qquad
\mu_f(t):=\mu(A_t(f)).
\end{equation*}
For the extremizer $f\equiv1$, write
\[
u_*(z):=(1-|z|^2)^\alpha,
\qquad
A_t^*:=\{u_*>t\},
\qquad
\mu_*(t):=\mu(A_t^*).
\]
For $0<t<1$,
\[
A_t^*=\{|z|<\sqrt{1-t^{1/\alpha}}\},
\]
so the model superlevel sets are centered Bergman balls.

\begin{lemma}[Coarea differential inequality]\label{lem:coarea_differential}
Suppose that the superlevel sets under consideration are relatively compact
and have finite invariant volume. Then, for almost every regular value $t>0$,
\begin{equation}\label{eq:global_coarea_differential}
 -\mu_f'(t)\ge
 \frac{P(A_t(f))^2}{4n\alpha\,t\,\mu_f(t)}.
\end{equation}
\end{lemma}

\begin{proof}
The coarea formula and Cauchy--Schwarz give
\[
 P(A_t)^2\le(-\mu_f'(t))
 \int_{u_f=t}|\nabla_bu_f|_b\,d\sigma_b.
\]
On $A_t$ one has $f\ne0$, and therefore
$\widetilde\Delta\log|f|=0$. Since
$\widetilde\Delta\log(1-|z|^2)=-4n$, the divergence theorem yields
\[
 \int_{u_f=t}|\nabla_bu_f|_b\,d\sigma_b
 =-t\int_{A_t}\widetilde\Delta\log u_f\,d\mu
 =4n\alpha t\mu_f(t).
\]
Substitution proves \eqref{eq:global_coarea_differential}.
\end{proof}

\subsection{Local distribution comparison and quantitative stability}
\label{sec:local_main}

\begin{definition}[Local centered interval hypothesis]\label{Setup}
Fix \(n\ge 1\), \(p>0\), \(\alpha>n\), and \(r_0>0\). Let \(f\) be holomorphic in \(\B_n\), write
\begin{align*}
   &f=1+\phi,\qquad u(z):=u_f(z)=|f(z)|^p(1-|z|^2)^\alpha,\qquad A_t:=\{u>t\},\\
   & \mu(t):=\mu(A_t),\qquad t_0:=\sup_{\B_n}u.
\end{align*}
Assume there exists a nontrivial interval of regular values
$0<t_-<t_+<\min\{t_0,1\}$ such that for a.e.\
\(t\in(t_-,t_+)\):
\begin{enumerate}[(i)]
\item if \(r(t)\) is defined by \(\mu(A_t)=\mu(\B_{r(t)})\), then \(r(t)\le r_0\);
\item there exists an automorphism \(\psi_t\in\Aut(\B_n)\) such that
\[
\widetilde A_t:=\psi_t(A_t)
\qquad\text{satisfies}\qquad
\mathrm{Bar}(\widetilde A_t)=0;
\]
\item \(\partial \widetilde A_t\) is a radial graph over \(\Sph\), in the parametrization
\eqref{eq:graph_param_thm}, with graph function \(\widetilde u_t\) satisfying
\[
\|\widetilde u_t\|_{W^{1,\infty}(\Sph)}\le \varepsilon_0(r_0),
\]
where \(\varepsilon_0(r_0)\) is the smallness threshold in
Theorem~\ref{thm:fuglede_bergman}.
\item the recenterings can be selected so that
$t\mapsto\widetilde u_t$ is strongly measurable as a map into
$W^{1,2}(\Sph)$.
\end{enumerate}
\end{definition}

\begin{remark}
When \(f=1+\phi\) is \(C^1\)-close to \(1\) on a compact subball and satisfies
the weighted tail condition \eqref{eq:tail_condition}, the local centered
interval hypothesis is verified by the radial-graph and recentering arguments;
see Appendix~\ref{app:nearly_spherical_verification}.
\end{remark}

\begin{theorem}[Local dominance near \(f\equiv 1\)]\label{thm:local_dominance}
Fix \(n\ge 1\), \(p>0\), \(\alpha>n\), and \(r_0>0\). Under the assumptions
of Setup~\ref{Setup}, assume in addition that \(t_-\) satisfies the anchoring inequality
\[
\mu(t_-)\le \mu_*(t_-),
\]
one has
\[
\mu(t)\le \mu_*(t)\qquad\text{for all }t\in[t_-,t_+].
\]
\end{theorem}

\begin{remark}[Status of the anchoring condition]\label{rem:anchoring_status}
The inequality $\mu(t_-)\le\mu_*(t_-)$ is an initial condition for the ODE
comparison and is not included in Setup~\ref{Setup}. If, in addition,
$\|f\|_{A^p_\alpha}=1$, then \eqref{eq:sharp_pointwise_bound} gives
$0\le u_f\le1$, and normalization yields the global identity
\[
 \int_0^1\bigl(\mu(t)-\mu_*(t)\bigr)\,dt=0.
\]
It follows that some level in $(0,1)$ satisfies
$\mu(t)\le\mu_*(t)$. However, compact $C^1$ control and the weighted tail
condition alone do not ensure that such a level belongs to the particular
compact interval on which the radial-graph hypotheses have been verified.
Accordingly, Theorem~\ref{thm:local_dominance} is conditional on the displayed
anchoring inequality; Proposition~\ref{prop:verify_setup} below verifies the
geometric hypotheses but does not claim to produce an anchor.
\end{remark}

\begin{corollary}[Truncated Bergman--Bergman comparison]\label{cor:local_contraction}
Assume the hypotheses of Theorem~\ref{thm:local_dominance}. For every \(q>p\) and \(\beta>\alpha\) with
\[
\frac{q}{p}=\frac{\beta}{\alpha},
\]
put $s=q/p$. Then
\begin{equation}\label{eq:local_contraction_new}
s\int_{t_-}^{t_+}t^{s-1}\mu(t)\,dt
\le s\int_{t_-}^{t_+}t^{s-1}\mu_*(t)\,dt.
\end{equation}
These are exactly the contributions of $[t_-,t_+]$ to the layer-cake
representations of
$\int_{\B_n}|f|^q(1-|z|^2)^\beta\,d\mu$ and
$\int_{\B_n}(1-|z|^2)^\beta\,d\mu$, respectively.  No comparison of the two
uncontrolled tails is asserted.
\end{corollary}

\begin{theorem}[Quantitative gap and shape control]\label{thm:local_gap}
Fix \(n\ge 1\), \(p>0\), \(\alpha>n\), and \(r_0>0\). Under the assumptions
of Setup~\ref{Setup}, and assuming $\mu(t_-)\le\mu_*(t_-)$, for every
\(q>p\) and \(\beta>\alpha\) with
\[
\frac{q}{p}=\frac{\beta}{\alpha},
\qquad
s:=\frac{q}{p}>1,
\]
there is a constant
\[
c=c(n,p,\alpha,q,r_0,t_-,t_+)>0
\]
such that the following quantitative truncated estimate holds:
\begin{equation}\label{eq:local_gap_new}
s\int_{t_-}^{t_+}t^{s-1}\mu(t)\,dt
\le s\int_{t_-}^{t_+}t^{s-1}\mu_*(t)\,dt
-c\int_{t_-}^{t_+}\frac{t_+^s-t^s}{t}
\|\widetilde u_t\|_{W^{1,2}(\Sph)}^2\,dt.
\end{equation}
\end{theorem}

\begin{remark}[Uniformity of the constants]\label{rem:uniform_constants}
For fixed dimension $n$, the constants in the preceding local results may be
chosen locally uniformly in the continuous parameters. More precisely, if
$(p,\alpha,r_0,t_-,t_+)$, together with $q$ in
Theorem~\ref{thm:local_gap}, ranges over a fixed compact subset of the
corresponding open admissible parameter set
\[
p>0,\qquad \alpha>n,\qquad r_0>0,\qquad
0<t_-<t_+<1,\qquad q>p,
\]
with $q/p=\beta/\alpha$ where relevant, then the smallness thresholds may be
chosen uniformly positive, and the constants in the estimates may be chosen
uniformly. This follows from the compact-radius estimates in the radial-graph,
re-centering, and Fuglede-type stability arguments, together with the explicit
one-dimensional estimate in the proof below. No uniformity is asserted as the
parameters approach the boundary of the admissible range, for instance as
$\alpha\downarrow n$, $p\downarrow0$, $q\downarrow p$, or $t_+\uparrow1$.
\end{remark}

\subsection{The Hardy bridge and the scope of the local comparison}\label{sec:endpoint_local_chain}
We record separately the quantitative information supplied by a compact level
interval and the Hardy endpoint limit. Neither statement supplies control of
the layer-cake tails outside the perturbative interval.

\begin{theorem}[Local moment comparison and Hardy limit]\label{thm:local_chain}
Fix $n\ge 1$, $p>0$, $\alpha>n$, and a radial-parameter cap $R_0>0$.
Let $r:=p/\alpha$, let $f$ be holomorphic in $\B_n$, and set
\[
v_f(z):=|f(z)|^{r}(1-|z|^2),
\qquad\text{so that}\qquad
v_f(z)^{\alpha}=|f(z)|^{p}(1-|z|^2)^{\alpha}=u_f(z).
\]
Assume that the hypotheses of Setup~\ref{Setup}, with radial-parameter cap $R_0$,
hold on some nontrivial interval of regular values $(t_-,t_+)$ for $u_f$.
Assume in addition the anchoring inequality
\[
\mu(t_-)\le\mu_*(t_-).
\]

Then, for every $q>p$ and $\beta>\alpha$ with $q/p=\beta/\alpha$, the
truncated comparison \eqref{eq:local_contraction_new} holds.  If, in addition,
$f\in H^{rn}(\B_n)$, then the Hardy endpoint is recovered in the critical limit:
\begin{equation*}\label{eq:hardy_limit_statement}
\lim_{\gamma\downarrow n}\ \|f\|_{A_\gamma^{r\gamma}}=\|f\|_{H^{rn}}.
\end{equation*}
\end{theorem}

\begin{proof}
The truncated assertion is Corollary~\ref{cor:local_contraction}.  To identify
the endpoint on the critical line $\alpha\downarrow n$, we invoke
Proposition~\ref{prop:bergman_to_hardy_limit}, which gives
\[
\lim_{\gamma\downarrow n}\ \|f\|_{A_\gamma^{r\gamma}}=\|f\|_{H^{rn}(\B_n)}
=\|f\|_{H^{nr}(\B_n)}.
\]
No monotonicity in $\gamma$ is needed for this limiting statement. In
particular, the critical limit should not be read as a comparison between a
fixed weighted Bergman norm and the Hardy norm.
\end{proof}
\begin{example}[Centrally symmetric superlevel sets]\label{ex:symmetric_superlevels}
Let $f$ be holomorphic in $\Bn$ and fix $r>0$. Assume that
\begin{equation*}\label{eq:mod_even}
|f(-z)|=|f(z)| \qquad \text{for all } z\in\Bn .
\end{equation*}
For each $t>0$ define the superlevel set
\[
A_t(f):=\{z\in\Bn:\ |f(z)|^{r}(1-|z|^2)>t\}.
\]
Then $A_t(f)$ is centrally symmetric for every $t$, namely
\[
A_t(f)=-A_t(f).
\]
Consequently, for every $t>0$ such that $0<\mu(A_t(f))<\infty$, the origin is
the holomorphic barycenter of $A_t(f)$ in the sense of
\cite{JacimovicKalajAnnFennMath}. Indeed, the change of variables
$z\mapsto -z$ and the evenness of the invariant density give
\[
\int_{A_t(f)} z\,(1-|z|^2)^{-n-1}\,d\nu(z)=0.
\]
Thus $0$ is a critical point of the barycenter functional $L_{A_t(f)}$.
Its uniqueness as a minimizer, proved in
\cite[Theorem~5.1]{JacimovicKalajAnnFennMath}, shows that the holomorphic
barycenter is $0$.
\end{example}

\subsection{A quantitative version of the comparison argument}\label{sec:quantitative_comparison_details}

We now give a streamlined proof of Theorems~\ref{thm:local_dominance} and
\ref{thm:local_gap}. The verification that the level sets of
\(
u(z)=|f(z)|^p(1-|z|^2)^\alpha
\)
are small radial graphs when \(f=1+\phi\) is \(C^1\)-close to \(1\) is
verified in Appendix~\ref{app:nearly_spherical_verification}. Under
Setup~\ref{Setup}, on the interval \((t_-,t_+)\) we may re-center each level set
\(
A_t=\{u>t\}
\)
by an automorphism \(\psi_t\in\Aut(\B_n)\) so that
\(
\widetilde A_t:=\psi_t(A_t)
\)
has Bergman barycenter \(0\), while
\[
\mu(\widetilde A_t)=\mu(A_t),\qquad P(\widetilde A_t)=P(A_t),
\]
and \(\partial \widetilde A_t\) remains in the small radial-graph regime required by
Theorem~\ref{thm:fuglede_bergman}.

\begin{proof}[Proof of Theorems~\ref{thm:local_dominance} and \ref{thm:local_gap}]
Fix a regular value \(t\in(t_-,t_+)\), and let
\[
\mu(t):=\mu(A_t),\qquad \mu_*(t):=\mu(\{(1-|z|^2)^\alpha>t\}).
\]
We also write \(r(t)\) for the unique radius such that
\[
\mu(A_t)=\mu(\mathbb B_{r(t)}).
\]

Lemma~\ref{lem:coarea_differential}, applied on the present level interval,
gives the basic differential inequality
\begin{equation}\label{eq:streamlined_diff}
-\mu'(t)\ge \frac{P(A_t)^2}{4n\alpha\,t\,\mu(t)}.
\end{equation}

We now use the geometric input. Since \(\widetilde A_t\) is centered and remains a small radial graph,
Theorem~\ref{thm:fuglede_bergman} applies to \(\widetilde A_t\), and by invariance under automorphisms,
\[
P(A_t)=P(\widetilde A_t)
\ge
S(r(t))\Bigl(1+c\,\|\widetilde u_t\|_{W^{1,2}(\Sph)}^2\Bigr),
\]
where \(S(r):=P(\mathbb B_r)\). After squaring and absorbing constants,
\[
P(A_t)^2
\ge
S(r(t))^2\Bigl(1+c\,\|\widetilde u_t\|_{W^{1,2}(\Sph)}^2\Bigr).
\]
Combining this with \eqref{eq:streamlined_diff} gives
\begin{equation}\label{eq:streamlined_diff_remainder}
-\mu'(t)\ge
\frac{S(r(t))^2}{4n\alpha\,t\,\mu(t)}
\Bigl(1+c\,\|\widetilde u_t\|_{W^{1,2}(\Sph)}^2\Bigr).
\end{equation}

Let
\[
V(r):=\mu(\mathbb B_r),
\qquad
\Phi(\xi):=\frac{S(V^{-1}(\xi))^2}{4n\alpha\,\xi}.
\]
Then \eqref{eq:streamlined_diff_remainder} becomes
\begin{equation}\label{eq:streamlined_phi}
-\mu'(t)\ge \frac{\Phi(\mu(t))}{t}
\Bigl(1+c\,\|\widetilde u_t\|_{W^{1,2}(\Sph)}^2\Bigr).
\end{equation}

For the extremizer \(f\equiv 1\), the superlevel sets are balls, so
\[
-\mu_*'(t)=\frac{\Phi(\mu_*(t))}{t}.
\]
Assume that \(t_-\) satisfies
\[
\mu(t_-)\le\mu_*(t_-).
\]
Now define
\[
G(x):=\int_x^{\mu(t_-)}\frac{d\xi}{\Phi(\xi)}.
\]
On the fixed compact interval of regular levels, the coarea formula implies
that $\mu$ is absolutely continuous; the model distribution $\mu_*$ is smooth.
Since $\Phi$ is positive and continuous on the corresponding compact volume
range, $G\circ\mu$ and $G\circ\mu_*$ are absolutely continuous there. Since
\(\Phi>0\), the function \(G\) is strictly decreasing. From
\eqref{eq:streamlined_phi} we obtain
\[
\frac{d}{dt}G(\mu(t))
=
-\frac{\mu'(t)}{\Phi(\mu(t))}
\ge
\frac1t
\Bigl(1+c\,\|\widetilde u_t\|_{W^{1,2}(\Sph)}^2\Bigr),
\]
whereas for the model profile,
\[
\frac{d}{dt}G(\mu_*(t))=\frac1t.
\]
Because $G$ is decreasing, the anchoring inequality gives
\[
G(\mu(t_-))=0\ge G(\mu_*(t_-)).
\]
Integrating from \(t_-\) to \(t\in[t_-,t_+]\),
we get
\[
G(\mu(t))\ge G(\mu_*(t)).
\]
Because \(G\) is strictly decreasing, this implies
\[
\mu(t)\le \mu_*(t)\qquad\text{for all }t\in[t_-,t_+],
\]
which proves Theorem~\ref{thm:local_dominance}.

Multiplying the distribution comparison by $s t^{s-1}$ and integrating on
$[t_-,t_+]$ proves Corollary~\ref{cor:local_contraction}.  Notice that this
step deliberately leaves the two layer-cake tails untouched.

Finally, keeping the positive remainder in \eqref{eq:streamlined_phi} gives
\[
G(\mu(t))-G(\mu_*(t))\ge c\int_{t_-}^{t}
\frac{\|\widetilde u_\tau\|_{W^{1,2}(\Sph)}^2}{\tau}\,d\tau.
\]
We next convert the left-hand side into a volume gap with a constant independent
of a possible small value of $\mu(t)$. Put
\[
\omega_{2n-1}:=\mathcal H(\Sph),\qquad C_n:=\frac{\omega_{2n-1}}{2n}.
\]
If $\rho=\tanh(r/2)$ and $x=\rho^2/(1-\rho^2)$, direct integration of the
invariant measure and the ball-perimeter formula give
\[
V(r)=C_nx^n,
\qquad
S(r)=\omega_{2n-1}x^{n-\frac12}(1+x)^{\frac12},
\qquad
\Phi(V(r))=\frac{\omega_{2n-1}}{2\alpha}x^{n-1}(1+x).
\]
Consequently, under the change of variables $\xi=C_nx^n$,
\[
\frac{d\xi}{\Phi(\xi)}=\alpha\frac{dx}{1+x}.
\]
For $t\in[t_-,t_+]$, define
\[
x(t):=\left(\frac{\mu(t)}{C_n}\right)^{1/n},
\qquad
x_*(t):=\left(\frac{\mu_*(t)}{C_n}\right)^{1/n}
=t^{-1/\alpha}-1.
\]
The already proved dominance gives $0\le x(t)\le x_*(t)$, while
\[
x_*(t)\ge a:=t_+^{-1/\alpha}-1>0.
\]
The preceding explicit formula therefore yields
\begin{align*}
0&\le G(\mu(t))-G(\mu_*(t))
=\alpha\log\frac{1+x_*(t)}{1+x(t)}
\le\alpha\bigl(x_*(t)-x(t)\bigr),\\
\mu_*(t)-\mu(t)
&=C_n\bigl(x_*(t)^n-x(t)^n\bigr)
\ge C_na^{n-1}\bigl(x_*(t)-x(t)\bigr).
\end{align*}
Hence, uniformly for $t\in[t_-,t_+]$,
\[
\mu_*(t)-\mu(t)
\ge \frac{C_na^{n-1}}{\alpha}
\bigl(G(\mu(t))-G(\mu_*(t))\bigr).
\]
Combining this estimate with the retained remainder, multiplying by
$s t^{s-1}$, and applying Tonelli's theorem yields
\[
s\int_{t_-}^{t_+}t^{s-1}\bigl(\mu_*(t)-\mu(t)\bigr)\,dt
\ge c\int_{t_-}^{t_+}\frac{t_+^s-t^s}{t}
\|\widetilde u_t\|_{W^{1,2}(\Sph)}^2\,dt.
\]
This is \eqref{eq:local_gap_new}.
\end{proof}

\begin{proposition}[Hardy limit of normalized weighted Bergman norms as $\alpha\downarrow n$]
\label{prop:bergman_to_hardy_limit}
Fix $n\ge1$ and $r>0$. Let $f$ be holomorphic in $\B_n$ and assume that
$f\in H^{rn}(\B_n)$.
Set $p(\alpha):=r\alpha$ and $\varepsilon:=\alpha-n>0$.

Let $dv_g$ be the invariant (Bergman) volume element. Define the normalized weighted probability measures
\[
d\mu_\alpha(z):=c_{\alpha,n}\,(1-|z|^2)^{\alpha}\,dv_g(z),
\qquad
c_{\alpha,n}:=\Bigl(\int_{\B_n}(1-|w|^2)^{\alpha}\,dv_g(w)\Bigr)^{-1},
\qquad \alpha>n,
\]
and the norms
\[
\|f\|_{A_\alpha^{p(\alpha)}}^{p(\alpha)}
:=\int_{\B_n} |f(z)|^{p(\alpha)}\,d\mu_\alpha(z).
\]

Then $f\in A_\alpha^{r\alpha}(\B_n)$ for every $\alpha>n$, and
\[
\lim_{\alpha\downarrow n}\ \|f\|_{A_\alpha^{r\alpha}}
=
\|f\|_{H^{rn}}.
\]
\end{proposition}

\begin{proof}
Put $s:=rn$ and
\[
\delta:=\frac{\alpha-n}{n}>0,
\qquad
X_f(z):=|f(z)|^s(1-|z|^2)^n.
\]
Then $r\alpha=s(1+\delta)$ and
\[
|f(z)|^{r\alpha}(1-|z|^2)^\alpha=X_f(z)^{1+\delta}.
\]
If $X_1(z):=(1-|z|^2)^n$, the definition of $c_{\alpha,n}$ gives the exact identity
\begin{equation}\label{eq:critical_ratio}
\|f\|_{A_\alpha^{r\alpha}}^{r\alpha}
=
\frac{\displaystyle\int_{\B_n}X_f^{1+\delta}\,dv_g}
{\displaystyle\int_{\B_n}X_1^{1+\delta}\,dv_g}.
\end{equation}

We first determine the small-level distribution of $X_f$.  Let
\[
D_f(t):=v_g(\{z\in\B_n:X_f(z)>t\}),\qquad t>0.
\]
The standard Hardy maximal theorem gives radial boundary values $f^*$ and an
integrable radial maximal function
\[
M_f(\zeta):=\sup_{0<\rho<1}|f(\rho\zeta)|^s\in L^1(\Sph).
\]
The pointwise Hardy estimate also shows that $X_f$ is bounded on $\B_n$.

Write $z=\rho\zeta$ and set $y=1-\rho^2$.  After absorbing the fixed
normalizing constants into $K_n>0$, polar coordinates give
\[
dv_g=K_n(1-y)^{n-1}y^{-n-1}\,dy\,d\sigma(\zeta).
\]
Making the change of variables $y=t^{1/n}x$, we obtain
\begin{align}
tD_f(t)
&=K_n\int_{\Sph}\int_0^{t^{-1/n}}
\mathbf 1_{\{x^n|f(\sqrt{1-t^{1/n}x}\,\zeta)|^s>1\}} \notag\\
&\hspace{35mm}\times(1-t^{1/n}x)^{n-1}x^{-n-1}\,dx\,d\sigma(\zeta).
\label{eq:scaled_distribution}
\end{align}
For almost every $\zeta$ and every fixed $x>0$, the integrand converges to
\[
\mathbf 1_{\{x^n|f^*(\zeta)|^s>1\}}x^{-n-1}.
\]
Dominated convergence applies on each compact $x$-interval.  The large-$x$
tail is bounded uniformly by $1/(nb^n)$ on $(b,\infty)$.  For the small-$x$
tail, the indicator in \eqref{eq:scaled_distribution} implies
$x^nM_f(\zeta)>1$, and hence
\[
\int_0^a\mathbf 1_{\{x^nM_f(\zeta)>1\}}x^{-n-1}\,dx
=\frac1n\bigl(M_f(\zeta)-a^{-n}\bigr)_+.
\]
Its integral tends to zero as $a\downarrow0$ because $M_f\in L^1(\Sph)$.
Letting first $t\downarrow0$ and then $a\downarrow0$, $b\uparrow\infty$, we obtain
\begin{equation}\label{eq:small_level_asymptotic}
\lim_{t\downarrow0}tD_f(t)
=\kappa_n\int_{\Sph}|f^*(\zeta)|^s\,d\sigma(\zeta)
=\kappa_n\|f\|_{H^s}^s,
\end{equation}
where $\kappa_n=K_n/n$.  In particular,
\begin{equation}\label{eq:model_small_level_asymptotic}
\lim_{t\downarrow0}tD_1(t)=\kappa_n.
\end{equation}
The same maximal-function estimate also gives $D_f(t)\le C/t$ for $t>0$.

We use the following elementary Abelian fact.  If $X\ge0$ is bounded,
$D_X(t)<\infty$ for $t>0$, and $tD_X(t)\to A$, then layer cake gives
\[
\delta\int X^{1+\delta}\,d\mu
=\delta(1+\delta)\int_0^{\|X\|_\infty}t^\delta D_X(t)\,dt
\longrightarrow A.
\]
Indeed, on $(0,\eta)$ this follows by sandwiching $tD_X(t)$ between
$A-\varepsilon$ and $A+\varepsilon$, since
$\delta\int_0^\eta t^{\delta-1}\,dt=\eta^\delta\to1$; the contribution of
$[\eta,\|X\|_\infty]$ tends to zero.  Notice that the factor $\delta$ is
essential.

Applying this fact to $X_f$ and $X_1$ and using
\eqref{eq:small_level_asymptotic}--\eqref{eq:model_small_level_asymptotic}, we find
\[
\lim_{\delta\downarrow0}
\frac{\int_{\B_n}X_f^{1+\delta}\,dv_g}
{\int_{\B_n}X_1^{1+\delta}\,dv_g}
=\|f\|_{H^s}^s.
\]
Moreover, boundedness of $X_f$ and the estimate $D_f(t)\le C/t$ show by layer
cake that
\[
\int_{\B_n} X_f^{1+\delta}\,dv_g<\infty
\qquad(\delta>0).
\]
Thus
$f\in A_\alpha^{r\alpha}$ for every $\alpha>n$.  Finally,
\eqref{eq:critical_ratio} and $r\alpha\to s$ yield
\[
\lim_{\alpha\downarrow n}\|f\|_{A_\alpha^{r\alpha}}
=\|f\|_{H^s}.
\]
\end{proof}

\subsection{Hinge reduction and the limitation of the local comparison}\label{sec:hinge_limitation}

Assume that $f=1+\phi$ with
$\|\phi\|_{C^1(\overline{\B_{\rho}})}\ll1$ for some Euclidean radius
$0<\rho<1$, under the hypotheses of Corollary~\ref{cor:local_contraction}, and
assume in addition that $\|f\|_{A^p_\alpha(\B_n)}=1$. Set
\[
w(z):=(1-|z|^2)^\alpha,\qquad U_f(z):=|f(z)|^p\,w(z),\qquad d\mu:=dv_g .
\]
The normalization $\|f\|_{A^p_\alpha(\B_n)}=1$ gives the mean identity
\begin{equation*}\label{eq:mean_match}
\int_{\B_n}U_f\,d\mu=\int_{\B_n}w\,d\mu.
\end{equation*}
Lemma~\ref{lem:sharp_pointwise} gives $0\le U_f\le1$, while $0\le w\le1$.
By Proposition~\ref{prop:hinge_reduction}, to prove Conjecture~\ref{conj:convex_functional} it suffices to show the hinge inequalities
\begin{equation}\label{eq:hinge_goal}
\int_{\B_n}(U_f-t)_+\,d\mu\le \int_{\B_n}(w-t)_+\,d\mu,\qquad t\in[0,1].
\end{equation}
Using $(x-t)_+=\int_t^\infty\mathbf 1_{\{x>s\}}\,ds$ we obtain the layer-cake formulas
\begin{equation}\label{eq:layer_cake}
\int_{\B_n}(U_f-t)_+\,d\mu=\int_t^1 \mu(A_s(f))\,ds,\qquad
A_s(f):=\{z\in\B_n:\ U_f(z)>s\},
\end{equation}
and similarly
\[
\int_{\B_n}(w-t)_+\,d\mu=\int_t^1 \mu(\B_{r_*(s)})\,ds,
\qquad (1-\tanh^2(r_*(s)/2))^\alpha=s.
\]
Thus \eqref{eq:hinge_goal} would follow from the distribution inequality
\emph{for all} $s\ge t$.  Theorem~\ref{thm:local_dominance}, however, proves only
\begin{equation}\label{eq:distribution_dominance}
\mu(A_s(f))\le \mu(\B_{r_*(s)})\qquad\text{for a.e.\ }s\in[t_-,t_+].
\end{equation}
This does not control the portions $[t,t_-)$ and $(t_+,1]$ of the
layer-cake integral when they are present.

Indeed, integrating \eqref{eq:distribution_dominance} yields only the
truncated comparison
\[
 \int_{t_-}^{t_+}\mu(A_s(f))\,ds
 \le \int_{t_-}^{t_+}\mu(\B_{r_*(s)})\,ds,
 \qquad (1-\tanh^2(r_*(s)/2))^\alpha=s.
\]
It supplies no information about the portions $[t,t_-)$ and
$(t_+,1]$ of the layer-cake formula. No second normalization of the model
radius is introduced here.
Thus the local hypotheses alone do not prove
Conjecture~\ref{conj:convex_functional}. Proposition~\ref{prop:hinge_reduction}
records the abstract hinge step and makes precise which additional global
distribution information would be needed.

\begin{proposition}[Hinge-function reduction]\label{prop:hinge_reduction}
Fix two measurable functions $X,Y:\B_n\to[0,1]$ such that $X,Y\in L^1(d\mu)$ and
\begin{equation}\label{eq:same_mean}
\int_{\B_n} X\,d\mu=\int_{\B_n} Y\,d\mu.
\end{equation}
For $t\in[0,1]$ define the \emph{hinge function}
\[
h_t(x):=(x-t)_+=\max\{x-t,0\}.
\]
Assume that for every $t\in[0,1]$ one has
\begin{equation}\label{eq:hinge_dom}
\int_{\B_n} h_t(X)\,d\mu \le \int_{\B_n} h_t(Y)\,d\mu.
\end{equation}
Then, for every convex $\Phi:[0,1]\to\R$ satisfying $\Phi(0)=0$ and
$\Phi(X),\Phi(Y)\in L^1(d\mu)$, we have
\begin{equation}\label{eq:convex_dom}
\int_{\B_n}\Phi(X)\,d\mu \le \int_{\B_n}\Phi(Y)\,d\mu.
\end{equation}
\end{proposition}

\begin{proof}
We first suppose that the one-sided derivatives of $\Phi$ at the endpoints
are finite. Its distributional second derivative $\Phi''$ is then a finite
nonnegative Borel measure on $(0,1)$, and the standard integral representation
for convex functions gives
\begin{equation}\label{eq:choquet}
\Phi(x)=\Phi'_+(0)\,x+\int_{(0,1)} (x-t)_+\,d\Phi''(t),
\qquad x\in[0,1].
\end{equation}
All terms are integrable because $(X-t)_+\le X$ and $(Y-t)_+\le Y$.
Tonelli's theorem, applied separately to the nonnegative hinge terms, and the
mean identity give
\begin{align*}
\int_{\B_n}\Phi(X)\,d\mu-\int_{\B_n}\Phi(Y)\,d\mu
&=
\Phi'_+(0)\Bigl(\int_{\B_n}X\,d\mu-\int_{\B_n}Y\,d\mu\Bigr)
\\
&\quad
+\int_{(0,1)}\Bigl(\int_{\B_n}(X-t)_+\,d\mu-\int_{\B_n}(Y-t)_+\,d\mu\Bigr)\,d\Phi''(t).
\end{align*}
The first term vanishes by the mean identity \eqref{eq:same_mean}, and the inner
difference in the second term is $\le0$ for every $t$ by \eqref{eq:hinge_dom},
while $d\Phi''\ge0$. Therefore the whole right-hand side is $\le0$.

For a general finite convex function with $\Phi(0)=0$, truncate a monotone
version of its derivative to $[-k,k]$ and normalize the resulting convex
function $\Phi_k$ by $\Phi_k(0)=0$. Then $\Phi_k\to\Phi$ pointwise on
$[0,1]$. Applying the preceding argument to $\Phi_k$, and then using the
assumed integrability together with the standard affine-minorant truncation
argument, proves \eqref{eq:convex_dom}.
\end{proof}

\paragraph{\bf Application to our setting.}
We apply the lemma with
\[
X(z)=U_f(z)=|f(z)|^p(1-|z|^2)^\alpha,
\qquad
Y(z)=w(z)=(1-|z|^2)^\alpha,
\]
which satisfy $\int X\,d\mu=\int Y\,d\mu$ by the normalization
$\|f\|_{A^p_\alpha}=1$. Hence proving \eqref{eq:hinge_goal}, i.e.
\[
\int_{\B_n}(U_f-t)_+\,d\mu \le \int_{\B_n}(w-t)_+\,d\mu \quad \forall\,t\in[0,1],
\]
implies
\[
\int_{\B_n}\Phi(U_f)\,d\mu \le \int_{\B_n}\Phi(w)\,d\mu
\]
for every admissible convex $\Phi:[0,1]\to\R$ with $\Phi(0)=0$.
The local interval comparison established in this paper does not by itself
verify this full family of hinge inequalities, because the two complementary
level ranges remain uncontrolled.

\begin{remark}[Subsequent development]\label{rem:subsequent_development}
The first version of the present work was completed before the later global
isoperimetric result of the first author. A subsequent preprint
\cite{KalajGromovRos2026} contains a global argument leading, in particular, to
a proof of Conjecture~\ref{conj:convex_functional}. We mention this only for
historical context. None of the statements or proofs in the present paper uses
that later global result; the contribution retained here is the original local
perturbative comparison, its quantitative stability remainder, and the
verification of the nearly spherical hypotheses.
\end{remark}

\appendix

\section{Verification of the local centered interval hypothesis}\label{app:nearly_spherical_verification}

In this appendix we justify the local centered interval hypothesis from
Setup~\ref{Setup} for perturbations \(f=1+\phi\) that are \(C^1\)-small on a
compact subball and satisfy the weighted tail condition below.
The argument has three parts:
\begin{enumerate}[(a)]
\item the level sets \(\{u=t\}\) are radial graphs for \(t\) in a compact interval of levels;
\item after re-centering by a small automorphism, these level sets remain in the small radial-graph class;
\item combining these facts with a bound on the Bergman barycenter yields Setup~\ref{Setup}.
\end{enumerate}

Throughout this appendix we write
\[
u(z):=|f(z)|^p(1-|z|^2)^\alpha,
\qquad f=1+\phi,
\qquad \alpha>n,
\qquad p>0.
\]

\subsection{Radial graph structure near the extremizer}\label{subsec:verify_nearly_spherical}

We first show that if \(f=1+\phi\) is \(C^1\)-close to \(1\), then on a compact interval of levels
the hypersurfaces \(\{u=t\}\) are Lipschitz radial graphs with small \(W^{1,\infty}\)-norm.

Let \(0<\rho<1\) and assume
\[
\|\phi\|_{C^1(\overline{\B_\rho})}\le \delta.
\]
Fix an interval of levels \(t\in[t_-,t_+]\) such that the corresponding radii for the extremizer satisfy
\[
0<\rho_- \le \rho_0(t)\le \rho_+<\rho,
\qquad
\rho_0(t):=\sqrt{1-t^{1/\alpha}}.
\]
We also assume
\begin{equation}\label{eq:tail_condition}
\sup_{|z|\ge\rho}|f(z)|^p(1-|z|^2)^\alpha<t_-.
\end{equation}
This natural weighted tail condition excludes additional superlevel components
outside the compact subball; compact $C^1$ control alone does not exclude them.

\begin{lemma}[Radial graph lemma]\label{lem:IFT_radial_graph}
There exists a constant
\[
\delta_{\mathrm{rg}}=\delta_{\mathrm{rg}}(n,p,\alpha,\rho,\rho_-,\rho_+)>0
\]
such that if
\[
\|\phi\|_{C^1(\overline{\B_\rho})}\le \delta_{\mathrm{rg}},
\]
and \eqref{eq:tail_condition} holds, then for every \(t\in[t_-,t_+]\) the
level set \(\{u=t\}\) is a Lipschitz radial graph:
\begin{equation}\label{eq:rho_graph}
\{u=t\}=\{\rho_t(\omega)\,\omega:\ \omega\in\Sn\},
\end{equation}
where \(\rho_t:\Sn\to(\rho_-/2,(\rho_++\rho)/2)\) satisfies
\begin{equation}\label{eq:rho_est}
\|\rho_t-\rho_0(t)\|_{W^{1,\infty}(\Sn)}
\le
C\,\|\phi\|_{C^1(\overline{\B_\rho})},
\end{equation}
with \(C=C(n,p,\alpha,\rho,\rho_-,\rho_+)\).
\end{lemma}

\begin{proof}
Fix \(t\in[t_-,t_+]\) and \(\omega\in\Sn\). For \(\rho\in(0,1)\), set
\[
F(\rho,\omega):=\log u(\rho\omega)-\log t
=
p\log|1+\phi(\rho\omega)|+\alpha\log(1-\rho^2)-\log t.
\]
Choose \(\delta_{\mathrm{rg}}\) so small that
\[
\|\phi\|_{C^0(\overline{\B_\rho})}\le \frac12.
\]
Then \(|1+\phi|\ge \frac12\) on \(\overline{\B_\rho}\), so \(\log|1+\phi|\) is well-defined and smooth.

For \(\phi\equiv 0\), the equation \(F(\rho,\omega)=0\) is independent of \(\omega\) and has the unique solution
\(\rho=\rho_0(t)\in[\rho_-,\rho_+]\).

Differentiating in \(\rho\):
\[
\partial_\rho F(\rho,\omega)
=
p\,\partial_\rho\bigl(\log|1+\phi(\rho\omega)|\bigr)
-\frac{2\alpha\rho}{1-\rho^2}.
\]
Let
\[
I:=[\rho_-/2,(\rho_++\rho)/2].
\]
For $s\in I$, the second term is bounded away from $0$, since the function
$s\mapsto s/(1-s^2)$ is increasing on $(0,1)$ and hence
\[
\frac{2\alpha s}{1-s^2}
\ge
\frac{\alpha\rho_-}{1-\rho_-^2/4}
=:2c_0.
\]
On the other hand,
\[
\bigl|\partial_\rho\log|1+\phi(\rho\omega)|\bigr|
\le
\frac{|\nabla\phi(\rho\omega)\cdot\omega|}{|1+\phi(\rho\omega)|}
\le
2\|\nabla\phi\|_{L^\infty(\overline{\B_\rho})}
\le
2\|\phi\|_{C^1(\overline{\B_\rho})}.
\]
Hence, after shrinking \(\delta_{\mathrm{rg}}\) if necessary so that \(2p\,\delta_{\mathrm{rg}}\le c_0\), we obtain
\begin{equation}\label{eq:drF_neg_appendix}
\partial_\rho F(\rho,\omega)\le -c_0<0
\qquad\text{for all }(\rho,\omega)\in I\times\Sn.
\end{equation}
Thus, for each \(\omega\), the map \(\rho\mapsto F(\rho,\omega)\) is strictly decreasing on \(I\).

Since
\[
F(\rho_0(t),\omega)=p\log|1+\phi(\rho_0(t)\omega)|,
\]
we have
\[
|F(\rho_0(t),\omega)|
\le
C\,\|\phi\|_{C^0(\overline{\B_\rho})}.
\]
After decreasing $\delta_{\mathrm{rg}}$, the values of $F$ at the endpoints of
$I$ have opposite signs, uniformly in $t$ and $\omega$.  Thus
\eqref{eq:drF_neg_appendix} and the intermediate value theorem give a unique
\(\rho_t(\omega)\in I\) such that \(F(\rho_t(\omega),\omega)=0\), and moreover
\begin{equation}\label{eq:rho_Linf_est_appendix}
|\rho_t(\omega)-\rho_0(t)|
\le
C\,\|\phi\|_{C^0(\overline{\B_\rho})}.
\end{equation}
It remains to exclude roots outside $I$. For the model function, compactness
of $[t_-,t_+]$ gives a uniform positive gap on
$[0,\rho_-/2]$ and a uniform negative gap on
$[(\rho_++\rho)/2,\rho]$. The perturbation term
$p\log|1+\phi|$ is uniformly $O(\delta_{\mathrm{rg}})$; after one final
decrease of $\delta_{\mathrm{rg}}$, these signs persist. The tail condition
\eqref{eq:tail_condition} excludes roots for $|z|\ge\rho$. Hence the root
found in $I$ is the only root on each ray, and this proves
\eqref{eq:rho_graph} for the entire level set.

Differentiating \(F(\rho_t(\omega),\omega)=0\) tangentially on \(\Sn\) yields
\[
\nabla_\omega \rho_t(\omega)
=
-\frac{\nabla_\omega F(\rho_t(\omega),\omega)}{\partial_\rho F(\rho_t(\omega),\omega)}.
\]
Since \(\log(1-\rho^2)\) is independent of \(\omega\),
\[
\nabla_\omega F(\rho,\omega)
=
p\,\nabla_\omega\bigl(\log|1+\phi(\rho\omega)|\bigr).
\]
Using \(|1+\phi|\ge \frac12\),
\[
\bigl|\nabla_\omega\log|1+\phi(\rho\omega)|\bigr|
\le
C\,|\nabla_\omega(\phi(\rho\omega))|
\le
C\,\rho\,\|\nabla\phi\|_{L^\infty(\overline{\B_\rho})}
\le
C\,\|\phi\|_{C^1(\overline{\B_\rho})}.
\]
Together with \eqref{eq:drF_neg_appendix}, this gives
\begin{equation}\label{eq:rho_W1inf_est_appendix}
\|\nabla_\omega \rho_t\|_{L^\infty(\Sn)}
\le
C\,\|\phi\|_{C^1(\overline{\B_\rho})}.
\end{equation}
Combining \eqref{eq:rho_Linf_est_appendix} and \eqref{eq:rho_W1inf_est_appendix} proves \eqref{eq:rho_est}.
\end{proof}

\begin{lemma}[Passage to the volume-matched radius]
\label{lem:volume_matched_radius}
Under the assumptions of Lemma~\ref{lem:IFT_radial_graph}, let $R_t$ be the
Euclidean radius determined by
\[
 \mu(A_t)=\mu(\{|z|<R_t\}).
\]
Then, uniformly for $t\in[t_-,t_+]$,
\begin{equation}\label{eq:volume_radius_estimate}
 |R_t-\rho_0(t)|
 +\left\|\frac{\rho_t}{R_t}-1\right\|_{W^{1,\infty}(\Sn)}
 \le C\|\phi\|_{C^1(\overline{\B_\rho})}.
\end{equation}
Equivalently, if $r(t)=2\operatorname{arctanh}R_t$ is the volume-matched
radial parameter, then $\partial A_t$ is a small Bergman radial graph relative
to $r(t)$, with the same order of $W^{1,\infty}$ control.
\end{lemma}

\begin{proof}
Put
\[
 W(R):=\mu(\{|z|<R\})
 =|\Sn|\int_0^R\frac{s^{2n-1}}{(1-s^2)^{n+1}}\,ds.
\]
The radial-subgraph description and the mean value theorem give
\begin{align*}
 |\mu(A_t)-W(\rho_0(t))|
 &\le C\int_{\Sn}|\rho_t(\omega)-\rho_0(t)|\,d\sigma(\omega)\\
 &\le C\|\phi\|_{C^1(\overline{\B_\rho})}.
\end{align*}
All radii remain in a fixed compact subinterval of $(0,1)$, on which
$W'$ is bounded above and bounded away from zero. Since
$W(R_t)=\mu(A_t)$, another application of the mean value theorem yields
$|R_t-\rho_0(t)|\le C\|\phi\|_{C^1}$. Combining this estimate with
\eqref{eq:rho_est}, and using that $R_t$ is uniformly bounded away from zero,
proves \eqref{eq:volume_radius_estimate}. The final assertion follows from
the smooth change of variables $R=\tanh(r/2)$ on compact radius intervals.
\end{proof}

\begin{remark}[From Euclidean to Bergman radial graphs]\label{rem:euclidean_to_bergman_graph}
Lemma~\ref{lem:IFT_radial_graph} gives a Euclidean radial graph \( |z|=\rho_t(\omega)\).
Since Bergman spheres are Euclidean spheres by radial symmetry, this is equivalent to a Bergman radial graph
after changing variables from Euclidean radius \(\rho\) to the radial parameter \(r\), via
\[
\rho=\tanh(r/2).
\]
Because this change of variables is smooth on compact radius intervals, estimate \eqref{eq:rho_est} implies the corresponding
smallness estimate for the Bergman graph function in \(W^{1,\infty}\).
\end{remark}

\subsection{Re-centering preserves the small graph class}

We next show that re-centering by a small automorphism preserves the small radial-graph regime needed to apply
Theorem~\ref{thm:fuglede_bergman}.

\begin{lemma}[Re-centering preserves small radial graphs]\label{lem:recenter_preserves_graph}
Fix \(0<r_-<r_+\le r_0\). There exist constants
\[
\eta_0=\eta_0(n,r_0,r_-,r_+)>0,
\qquad
C=C(n,r_0,r_-,r_+)\ge 1,
\]
such that the following holds.

Let \(E\subset\B_n\) have boundary given by a Bergman radial graph as in \eqref{eq:graph_param_thm},
with graph function \(u_E\), and suppose
\[
\|u_E\|_{W^{1,\infty}(\Sph)}\le \eta_0,
\qquad
\mu(E)=\mu(\B_r)\ \text{for some }r\in[r_-,r_+],
\qquad
a:=\mathrm{Bar}(E)\ \text{satisfies }|a|\le \eta_0.
\]
Let \(\psi_a\in\Aut(\B_n)\) satisfy \(\psi_a(a)=0\), and set
\[
\widetilde E:=\psi_a(E).
\]
Then \(\partial\widetilde E\) is again a Bergman radial graph as in \eqref{eq:graph_param_thm}, with graph function
\(\widetilde u_E\) satisfying
\begin{equation}\label{eq:recenter_W1inf_bound_fixed_appendix}
\|\widetilde u_E\|_{W^{1,\infty}(\Sph)}
\le
C\bigl(\|u_E\|_{W^{1,\infty}(\Sph)}+|a|\bigr),
\end{equation}
and
\begin{equation}\label{eq:recenter_W12_bound_fixed_appendix}
\|\widetilde u_E\|_{W^{1,2}(\Sph)}
\le
C\bigl(\|u_E\|_{W^{1,2}(\Sph)}+|a|\bigr).
\end{equation}
In particular, if \(\|u_E\|_{W^{1,\infty}}+|a|\) is sufficiently small, then
\[
\|\widetilde u_E\|_{W^{1,\infty}(\Sph)}\le \varepsilon_0(r_0),
\]
so Theorem~\ref{thm:fuglede_bergman} applies to \(\widetilde E\).
\end{lemma}

\begin{proof}
We give the details in Euclidean radial coordinates; the passage to Bergman
radius is made at the end. Since $r\in[r_-,r_+]$ and $u_E$ is small, there
are fixed numbers $0<\rho_1<\rho_2<1$ such that
$\partial E\subset\{\rho_1<|z|<\rho_2\}$. Write
\[
 Z(\omega)=\rho_E(\omega)\omega,
 \qquad \omega\in\Sph,
\]
where the smooth change between Euclidean radius and the radial parameter gives
\begin{equation}\label{eq:euclidean_graph_control}
 \|\rho_E-\rho_r\|_{W^{1,j}(\Sph)}
 \le C\|u_E\|_{W^{1,j}(\Sph)},\qquad j\in\{2,\infty\}.
\end{equation}

Choose the canonical automorphism $\psi_a$ given explicitly in
\eqref{eq:canonical_automorphism} below. The identity
$(1-\sqrt{1-|a|^2})/|a|^2=(1+\sqrt{1-|a|^2})^{-1}$ removes the apparent
singularity of the projection formula at $a=0$. Its first two derivatives imply,
uniformly on $\{|z|\le(1+\rho_2)/2\}$,
\begin{equation}\label{eq:automorphism_C2_control}
 \|\psi_a-\operatorname{Id}\|_{C^2}\le C|a|.
\end{equation}
Set $Y_a(\omega):=\psi_a(Z(\omega))$ and
$\Theta_a(\omega):=Y_a(\omega)/|Y_a(\omega)|$. From
\eqref{eq:euclidean_graph_control}--\eqref{eq:automorphism_C2_control},
\[
 \|\Theta_a-\operatorname{Id}\|_{C^1(\Sph)}
 \le C\bigl(\|u_E\|_{W^{1,\infty}}+|a|\bigr).
\]
If $\eta_0$ is sufficiently small, the quantitative inverse function theorem
on the compact sphere shows that $\Theta_a$ is a $C^1$ diffeomorphism and
that
\[
 \|\Theta_a^{-1}-\operatorname{Id}\|_{C^1}
 \le C\bigl(\|u_E\|_{W^{1,\infty}}+|a|\bigr).
\]
Consequently every ray meets $\psi_a(\partial E)$ exactly once, and its
Euclidean radial function is
\[
 \widetilde\rho(\theta)
 :=|Y_a(\Theta_a^{-1}(\theta))|.
\]
The chain rule, the preceding bounds, and
\eqref{eq:automorphism_C2_control} give
\[
 \|\widetilde\rho-\rho_r\|_{W^{1,\infty}}
 \le C(\|u_E\|_{W^{1,\infty}}+|a|),
 \qquad
 \|\widetilde\rho-\rho_r\|_{W^{1,2}}
 \le C(\|u_E\|_{W^{1,2}}+|a|).
\]
The maps between Euclidean radii and radial parameters have uniformly bounded first
two derivatives on the fixed annulus. Applying the chain rule once more gives
\eqref{eq:recenter_W1inf_bound_fixed_appendix} and
\eqref{eq:recenter_W12_bound_fixed_appendix}. Finally, invariance of $dv_g$
under automorphisms gives $\mu(\widetilde E)=\mu(E)$, so the comparison radius
is still $r$.
\end{proof}

\subsection{Conclusion: verification of the local setup}

The remaining ingredient is quantitative control of the Bergman barycenter.
The following lemma supplies the estimate needed to combine the two preceding
radial-graph arguments.

\begin{lemma}[Barycenter bound on the perturbative interval]\label{lem:barycenter_bound_appendix}
Assume that \(f=1+\phi\) with \(\|\phi\|_{C^1(\overline{\B_\rho})}\) sufficiently small, where
\(0<\rho<1\), and let
\[
u(z)=|f(z)|^p(1-|z|^2)^\alpha,
\qquad
A_t:=\{u>t\}.
\]
Assume moreover that \(t\in[t_-,t_+]\) lies in the compact level interval from
Lemma~\ref{lem:IFT_radial_graph}, so that
\[
0<\rho_-\le \rho_0(t)\le \rho_+<\rho,
\qquad
\rho_0(t):=\sqrt{1-t^{1/\alpha}},
\]
and \(\partial A_t\) is given by the radial graph
\[
\partial A_t=\{\rho_t(\omega)\omega:\omega\in\Sn\},
\qquad
\|\rho_t-\rho_0(t)\|_{W^{1,\infty}(\Sn)}
\le C\|\phi\|_{C^1(\overline{\B_\rho})}.
\]
Let \(\mathrm{Bar}(A_t)\) denote the holomorphic barycenter of \(A_t\) in the sense of
Ja\'cimovi\'c--Kalaj. Then one has
\[
|\mathrm{Bar}(A_t)|
\le
C\,\|\phi\|_{C^1(\overline{\B_\rho})}
\qquad\text{for all }t\in[t_-,t_+],
\]
where \(C=C(n,p,\alpha,\rho_-,\rho_+,\rho)\).
\end{lemma}

\begin{proof}
For a measurable set \(E\subset\B_n\) with \(0<\mu(E)<\infty\), let
\[
L_E(a):=
-\int_E
\log\frac{(1-|a|^2)(1-|z|^2)}{|1-\langle a,z\rangle|^2}\,d\mu(z),
\qquad a\in\B_n.
\]
By Theorem~5.1 and Definition~5.3 of Ja\'cimovi\'c--Kalaj \cite{JacimovicKalajAnnFennMath},
the holomorphic barycenter \(\mathrm{Bar}(E)\) is the unique minimizer of \(L_E\).
Applying this with \(E=A_t\), we write
\[
L_t(a):=L_{A_t}(a),
\qquad
a(t):=\mathrm{Bar}(A_t).
\]
Then \(a(t)\) is the unique critical point of \(L_t\), hence
\[
\nabla L_t(a(t))=0.
\]

We first record the gradient at the origin. For every real direction
$\xi\in\C^n\simeq\R^{2n}$, differentiation under the integral sign gives
\[
 DL_t(0)[\xi]=-2\operatorname{Re}
 \left\langle\xi,\int_{A_t}z\,d\mu(z)\right\rangle.
\]
Thus, under the usual real identification of $\C^n$ with $\R^{2n}$,
\begin{equation}\label{eq:gradLt0}
\nabla L_t(0)=-2\int_{A_t} z\,d\mu(z).
\end{equation}

We next estimate the first moment of \(A_t\).
Since \(A_t\) is the radial subgraph
\[
A_t=\{r\omega:\ 0\le r<\rho_t(\omega),\ \omega\in\Sn\},
\]
polar coordinates give
\[
\int_{A_t} z\,d\mu(z)
=
|\Sn|\int_{\Sn}\omega
\left(\int_0^{\rho_t(\omega)} \frac{r^{2n}}{(1-r^2)^{n+1}}\,dr\right)d\sigma(\omega),
\]
where $d\sigma$ is the normalized surface measure on \(\Sn\).
Set
\[
F(r):=\int_0^r \frac{s^{2n}}{(1-s^2)^{n+1}}\,ds.
\]
Then
\[
\int_{A_t} z\,d\mu(z)
=
|\Sn|\int_{\Sn}\omega\,F(\rho_t(\omega))\,d\sigma(\omega).
\]
Since \(\int_{\Sn}\omega\,d\sigma(\omega)=0\), we may subtract the constant \(F(\rho_0(t))\):
\[
\int_{A_t} z\,d\mu(z)
=
|\Sn|\int_{\Sn}\omega\bigl(F(\rho_t(\omega))-F(\rho_0(t))\bigr)\,d\sigma(\omega).
\]
Now \(\rho_t(\omega)\) and \(\rho_0(t)\) stay in the fixed compact interval
\(
[\rho_-/2,(\rho_++\rho)/2]
\),
so \(F'\) is bounded there. Hence, by the mean value theorem,
\[
|F(\rho_t(\omega))-F(\rho_0(t))|
\le
C\,|\rho_t(\omega)-\rho_0(t)|
\le
C\,\|\phi\|_{C^1(\overline{\B_\rho})}.
\]
Therefore
\begin{equation}\label{eq:first_moment_bound}
\left|\int_{A_t} z\,d\mu(z)\right|
\le
C\,\|\phi\|_{C^1(\overline{\B_\rho})}.
\end{equation}
Combining \eqref{eq:gradLt0} and \eqref{eq:first_moment_bound} yields
\begin{equation}\label{eq:gradLt0_bound}
|\nabla L_t(0)|
\le
C\,\|\phi\|_{C^1(\overline{\B_\rho})}.
\end{equation}

We now use the nondegeneracy of the barycenter equation near a centered ball.
For fixed \(z\in\B_n\), set
\[
\Psi_z(a):=
-\log\frac{(1-|a|^2)(1-|z|^2)}{|1-\langle a,z\rangle|^2}.
\]
For a real direction $\xi\in\C^n\simeq\R^{2n}$, direct differentiation at
$a=0$ gives the corrected formula
\begin{equation}\label{eq:correct_hessian_integrand}
D^2\Psi_z(0)[\xi,\xi]
=2|\xi|^2-2\operatorname{Re}\bigl(\langle\xi,z\rangle^2\bigr).
\end{equation}
The second term is not sign-definite pointwise.  For a centered ball $B$,
however, its integral vanishes by invariance under $z\mapsto iz$, and hence
\begin{equation}\label{eq:ball_hessian}
D^2L_B(0)[\xi,\xi]=2\mu(B)|\xi|^2.
\end{equation}
The model radii remain in a compact interval. The radial-graph estimate and
continuity of the first two derivatives therefore give constants $a_0,c_0>0$,
uniform in $t$, such that
\begin{equation}\label{eq:local_hessian_bound}
D^2L_t(a)[\xi,\xi]\ge c_0|\xi|^2\qquad (|a|\le a_0).
\end{equation}
We justify that the minimizer lies in this neighborhood. All the sets $A_t$
are contained in the fixed compact ball $\overline{\B_\rho}$ and converge in
symmetric difference, uniformly in $t$, to the corresponding centered balls
as $\|\phi\|_{C^1}\to0$. Hence $L_t$ and its first two derivatives converge
uniformly on compact subsets of the $a$-ball to their centered-ball
counterparts. Moreover, for $z\in\overline{\B_\rho}$,
\[
 |1-\langle a,z\rangle|\ge1-\rho,
\]
so the definition of $L_t$ gives, uniformly in $t$,
\[
 L_t(a)\ge-\mu(A_t)\log(1-|a|^2)-C.
\]
The volumes $\mu(A_t)$ are bounded below on the compact level interval.
Therefore $L_t(a)\to+\infty$ uniformly as $|a|\to1$, and the minimizers stay
in a fixed compact subset of $\B_n$. If a sequence of perturbations tended to
zero while $|a(t)|\ge a_0$, compactness and uniform convergence would produce
a nonzero minimizer of a centered-ball functional, contradicting the
uniqueness in \cite[Theorem~5.1]{JacimovicKalajAnnFennMath}. After decreasing
the perturbation threshold, we thus have $|a(t)|<a_0$. Integrating
\eqref{eq:local_hessian_bound} along the segment from $0$ to $a(t)$ and using
$\nabla L_t(a(t))=0$ gives
\[
c_0|a(t)|\le |\nabla L_t(0)|.
\]
Using \eqref{eq:gradLt0_bound}, we arrive at
\[
|a(t)|
\le
C\,\|\phi\|_{C^1(\overline{\B_\rho})}.
\]
Since \(a(t)=\mathrm{Bar}(A_t)\), the proof is complete.
\end{proof}

\begin{lemma}[Continuous choice of the recentering]\label{lem:measurable_recentering}
Under the hypotheses of Lemmas~\ref{lem:IFT_radial_graph} and
\ref{lem:barycenter_bound_appendix}, the maps
\[
 t\longmapsto \rho_t\in C^1(\Sn),
 \qquad t\longmapsto a(t):=\operatorname{Bar}(A_t)\in\B_n
\]
are continuous on $[t_-,t_+]$. One may choose the recentering automorphism
$\psi_t=\psi_{a(t)}$ continuously in $t$. The resulting graph function
$t\mapsto\widetilde u_t$ is continuous in $W^{1,2}(\Sph)$; in particular,
$t\mapsto\|\widetilde u_t\|_{W^{1,2}}^2$ is measurable.
\end{lemma}

\begin{proof}
The equation $F(\rho_t(\omega),\omega,t)=0$ in the proof of
Lemma~\ref{lem:IFT_radial_graph} has a radial derivative bounded away from
zero uniformly in $(t,\omega)$. The parameterized implicit-function theorem,
with the estimates already proved there, therefore gives continuity of
$t\mapsto\rho_t$ in $C^1(\Sn)$. It follows that
$t\mapsto\mathbf1_{A_t}$ is continuous in $L^1(d\mu)$.

The functionals $L_{A_t}$ are uniformly coercive and depend continuously on
$t$, locally uniformly together with their first derivatives. Their minimizer
is unique by \cite[Theorem~5.1]{JacimovicKalajAnnFennMath}. The usual compactness
argument for unique minimizers consequently gives continuity of $a(t)$.

For clarity, we use the canonical automorphism
\begin{equation}\label{eq:canonical_automorphism}
 \psi_a(z)=\frac{s_a z-a+(1-s_a)P_a z}{1-\langle z,a\rangle},
 \qquad s_a=\sqrt{1-|a|^2},
\end{equation}
where $P_a z=\langle z,a\rangle a/|a|^2$ for $a\ne0$ and the last term is
defined by continuity at $a=0$. Indeed,
$(1-s_a)/|a|^2=1/(1+s_a)$, so the apparently singular term equals
\[
 \frac{\langle z,a\rangle a}{1+s_a}
\]
and is smooth in $a$. Thus $\psi_a(a)=0$, $\psi_0=\operatorname{Id}$, and
$a\mapsto\psi_a$ is smooth in $C^2$ on every compact subball. For sufficiently
small perturbations, $A_t$ contains a fixed ball while $|a(t)|$ is small;
hence $a(t)\in A_t$ and $0\in\psi_{a(t)}(A_t)$. Applying the explicit graph
construction in Lemma~\ref{lem:recenter_preserves_graph} with the continuous
parameters $\rho_t$ and $a(t)$ proves continuity of $\widetilde u_t$ in
$C^1$, and therefore in $W^{1,2}$.
\end{proof}

\begin{proposition}[Verification of the local centered interval hypothesis]\label{prop:verify_setup}
Fix \(n\ge 1\), \(p>0\), \(\alpha>n\), a Euclidean radius
$0<\rho<1$, and numbers $0<\rho_-<\rho_+<\rho$. There exists
\[
\delta_0=\delta_0(n,p,\alpha,\rho,\rho_-,\rho_+)>0
\]
such that the following holds.

Let \(f=1+\phi\) be holomorphic in \(\B_n\). Assume that there are numbers
$0<t_-<t_+<1$ such that
\[
\rho_-\le \rho_0(t):=\sqrt{1-t^{1/\alpha}}\le\rho_+
\qquad(t\in[t_-,t_+]),
\]
and suppose that
\[
\|\phi\|_{C^1(\overline{\B_{\rho}})}\le \delta_0.
\]
Assume in addition that this interval satisfies
\[
\sup_{|z|\ge\rho}|f(z)|^p(1-|z|^2)^\alpha<t_-.
\]
Then, after decreasing $\delta_0$ if necessary, the geometric and measurability
parts of Setup~\ref{Setup} hold on
\((t_-,t_+)\) for a suitable radial-parameter cap
$R_0>0$.
\end{proposition}

\begin{proof}
By hypothesis, the compact interval of levels $[t_-,t_+]$ has model radii
$\rho_0(t)$ in the fixed interval $[\rho_-,\rho_+]\subset(0,\rho)$.
By Lemma~\ref{lem:IFT_radial_graph}, after taking \(\delta_0\) sufficiently small, each level set \(\{u=t\}\),
\(t\in[t_-,t_+]\), is a Euclidean radial graph with
\[
\|\rho_t-\rho_0(t)\|_{W^{1,\infty}(\Sn)}
\le
C\|\phi\|_{C^1(\overline{\B_{\rho}})}.
\]
Lemma~\ref{lem:volume_matched_radius} and
Remark~\ref{rem:euclidean_to_bergman_graph} yield the Bergman radial-graph
description of \(\partial A_t\) relative to its volume-matched radius, with
small \(W^{1,\infty}\)-norm. Since the level interval is compact away from \(0\) and \(\sup u\), the associated
radial parameters \(r(t)\) remain bounded by some $R_0>0$.

Next, Lemma~\ref{lem:barycenter_bound_appendix} gives
\[
|a(t)|:=|\mathrm{Bar}(A_t)|
\le
C\|\phi\|_{C^1(\overline{\B_{\rho}})}
\qquad\text{for a.e. }t\in(t_-,t_+).
\]
Applying Lemma~\ref{lem:recenter_preserves_graph} with \(E=A_t\) and \(a=a(t)\), we conclude that the re-centered sets
\[
\widetilde A_t:=\psi_t(A_t),
\qquad
\psi_t(a(t))=0,
\]
still have boundary given by a Bergman radial graph, with graph function \(\widetilde u_t\) satisfying
\[
\|\widetilde u_t\|_{W^{1,\infty}(\Sph)}
\le
\varepsilon_0(R_0)
\]
provided \(\delta_0\) is chosen sufficiently small. By construction,
\[
\mathrm{Bar}(\widetilde A_t)=0.
\]
Lemma~\ref{lem:measurable_recentering} shows that the recenterings and their
graph functions may be chosen continuously in $t$. Thus all geometric and
measurability parts of Setup~\ref{Setup} are verified. As emphasized in
Remark~\ref{rem:anchoring_status}, the separate anchoring inequality required
by Theorem~\ref{thm:local_dominance} is not asserted here.
\end{proof}
\section*{Acknowledgements}
The first author gratefully acknowledges financial support from the Ministry of Education, Science and Innovation of Montenegro through the grants \emph{``Mathematical Analysis, Optimisation and Machine Learning''} and \emph{``Complex-analytic and geometric techniques for non-Euclidean machine learning: theory and applications.''} Jian-Feng Zhu was supported by the National Natural Science Foundation of China (Grant Nos. 12271189 and 12671096), the Natural Science Foundation of Guangdong Province (Grant Nos. 2024A1515010467 and 2026A1515012333), the STU Scientific Research Initiation Grant (No. NTF25017T), the HQU Teaching Reform Project (No. HQJGKT2411), and the Fujian Alliance of Mathematics (Grant No. 2023SXLMMS07).

\end{document}